%% file: main.tex
\documentclass[12pt, reqno]{amsart}
\usepackage[utf8]{inputenc}
\usepackage{amsfonts,amsmath,amssymb,amscd,amsthm,url,enumerate, bm}
\usepackage{graphicx,epsf,afterpage}
\usepackage{multicol}
\usepackage{color}
\usepackage[german, english]{babel}
\usepackage[
    draft = false,
    unicode = true,
    colorlinks = true,
    allcolors = blue,
    hyperfootnotes = true
]{hyperref}

\usepackage{graphicx,xcolor}
\graphicspath{{figures/}} % путь, где искать картинки
\newcommand{\executeiffilenewer}[3]{%
  \ifnum\pdfstrcmp{\pdffilemoddate{#1}}%
   {\pdffilemoddate{#2}} > 0 {\immediate\write18{#3}}\fi}

\theoremstyle{plain}
\newtheorem{theorem}{Theorem}[]

\newtheorem{oldtheorem}{Theorem}

\newtheorem{lemma}[]{Lemma}

\theoremstyle{definition}
\newtheorem{remark}[]{Remark}
\newtheorem*{ackn}{Acknowledgement}

\newcommand{\blalpha}{\boldsymbol{\alpha}}
\newcommand{\blsigma}{\boldsymbol{\sigma}}

\newcommand{\knm}{\mathfrak{k}}
\newcommand{\T}{\mathfrak{T}}
\newcommand{\func}[1]{\psi_{\alpha_{#1}}}
\newcommand{\Z}{\mathbb{Z}}
\newcommand{\contfrac}[1]{[#1_0;#1_1, #1_2, \dots, #1_\nu, \dots]}

\def\R{{\mathbb R}} \def\Z{{\mathbb Z}}

\begin{document}
\title{On irrationality measure functions for several real numbers}
\author{Viktoria Rudykh}
\date{}

\begin{abstract}
	\input{thesys}
	
\end{abstract}
\maketitle

\section{Introduction}
For an irrational number $\xi \in \R$ we consider the irrationality measure function
\begin{equation*}
	\psi_\xi(t) = \min\limits_{1 \leq q \leq t, \ q \in \Z} || q\xi||,
\end{equation*}
where $|| . ||$ denotes the distance to the nearest integer.

Let 
\begin{equation*}
	q_0 \leq q_1 < q_2 < \dots < q_n < q_{n+1} < \dots
\end{equation*}
be the sequence of denominators of convergents to $\xi$. It is a well known fact (see \cite{Schmidt}) that
\begin{equation*}
	\psi_\xi(t) = || q_n\xi|| \text{ for } q_n \leq t < q_{n+1}.
\end{equation*}
In 2010 Kan and Moshchevitin \cite{KM} proved
\begin{oldtheorem}
	\label{th:KM}
	For any two  different irrational numbers $\alpha, \beta$ such that $\alpha \pm \beta \not\in \Z$ the difference function
	\begin{equation*}
		\psi_\alpha(t)  - \psi_\beta(t)
	\end{equation*}
	changes its sign infinitely many times as $t \rightarrow +\infty$.
\end{oldtheorem}
This result was generalised by several authors (see \cite{Dubickas}, \cite{M}, \cite{Shatskov}, \cite{Shulga}).

We call two irrational numbers $\alpha, \beta \in \R$ \textit{independent} if
\begin{equation}
	\label{eq:independent}
	\psi_\alpha(t) \neq \psi_\beta(t)
\end{equation}
for all $t$ large enough.

%an n-tuple?, запятая после definition
Let $\blalpha = (\alpha_1, \dots, \alpha_n) $ be an $n$-tuple of pairwise independent numbers. By definition (\ref{eq:independent}) we have $$\func{i}(t) \neq \func{j}(t)$$ for all $t$ large enough and for all $i \neq j$. So for large enough $t$ we have well defined permutation
\begin{equation*}
	\blsigma(t): \{1, 2, 3, \dots, n\} \rightarrow \{ \sigma_1, \sigma_2, \sigma_3, \dots, \sigma_n \},
\end{equation*}
\begin{equation*}
	\func{\sigma_1}(t) > \func{\sigma_2}(t) > \func{\sigma_3}(t) > \ldots > \func{\sigma_n}(t).
\end{equation*}

We define $\knm$-index by  
\begin{multline}
	\label{eq:k defin}
	\knm(\blalpha) = \max\{k : \exists \text{ different permutations } \blsigma_1, \dots, \blsigma_k: \\ \forall j \in \{1, \dots, k\}\ \forall t_0 > 0 \ \exists t > t_0 \ \blsigma(t)=\blsigma_j\}.
\end{multline}
So $\knm(\blalpha)$ is the number of permutations, which occur infinitely many times as $t \rightarrow \infty$.

In 2021 Manturov and Moshchevitin \cite{MM} proved the following statements. 
\begin{oldtheorem}
	\label{th:old k}
	For $n$-tuple $\blalpha = (\alpha_1, \dots, \alpha_n)$ of pairwise independent numbers one has
	\begin{equation*}
		\knm(\blalpha) \geq \sqrt{\frac{n}{2}}.
	\end{equation*}
\end{oldtheorem}
\begin{oldtheorem}
	\label{th:MM}
	For $k \geq 3$ and $n = \frac{k(k+1)}{2}$ there exists an $n$-tuple $\blalpha$ of pairwise independent numbers with
	\begin{equation*}
		\knm(\alpha) = k.
	\end{equation*}
\end{oldtheorem}

In this article we give an improvement of Theorem \ref{th:old k}. Now we formulate our main result. 
\begin{theorem}
	\label{th:main}
	%is???
	The length of $n$-tuple $\blalpha = (\alpha_1, \dots, \alpha_n)$ of pairwise independent numbers with $\knm(\blalpha) = k$ is
	\begin{equation*}
	n \leq \frac{k(k+1)}{2}.
	\end{equation*}
\end{theorem}
The bound from Theorem \ref{th:main} is optimal because of Theorem \ref{th:MM}.
\section{Auxiliary lemmas}
We represent irrational numbers $\alpha$ and $\beta$ as continued fractions
\begin{gather*}
	\alpha = \contfrac{a},\\
	 \beta = \contfrac{b},
\end{gather*}
where $a_0, b_0 \in \Z; \ a_\nu, b_\nu \in \Z_+, \ \nu = 1, 2, 3, \dots.$ \\
We define
\begin{align*}
	\alpha_\nu &= [a_\nu;a_{\nu+1}, a_{\nu+2}, \dots],  &\beta_\mu &= [b_\mu;b_{\mu+1}, b_{\mu+2}, \dots],\\
	\alpha_\nu^* &= [0; a_\nu, a_{\nu-1}, a_{\nu-2}, \dots, a_1], &\beta_\mu^* &= [0; b_\mu, b_{\mu-1}, b_{\mu-2}, \dots, b_1],\\
	\frac{p_\nu}{q_\nu} &= [a_0; a_1, a_2, \dots, a_\nu], &\frac{s_\mu}{r_\mu} &= [b_0; b_1, b_2, \dots, b_\mu],\\
	\xi_\nu &= |q_\nu \alpha - p_\nu|, &\eta_\mu &= |r_\mu \beta - s_\mu|.
\end{align*}
So,
\begin{gather*}
	\psi_\alpha(t) = \xi_\nu, \ q_\nu \leq t < q_{\nu+1}, \\
	\psi_\beta(t) = \eta_\mu, \ r_\mu \leq t < r_{\mu+1},
\end{gather*}
and
\begin{equation*}
	\alpha_\nu^* = \frac{q_{\nu-1}}{q_\nu}, \ \beta_\mu^* = \frac{r_{\mu-1}}{r_\mu}.
\end{equation*}
The next two simple lemmas describe obvious property of continued fractions (see \cite{M}, \cite{MM}).
\begin{lemma}
	\label{lm:pairs}
	If there are infinitely many pairs $(q_\nu, q_{\nu+1}) = (r_\mu, r_{\mu+1})$, then $\alpha \pm \beta \in \Z$.
\end{lemma}
\begin{lemma}
	\label{lm:tails}
	If $\alpha_\nu^* = \beta_\mu^*$, then $q_{\nu-1} = r_{\mu-1}, \ q_\nu = r_\mu$.
\end{lemma} 
From Lemma \ref{lm:pairs} and Lemma \ref{lm:tails} we immediately obtain the following result.
\begin{lemma}
	\label{lm:independent}
	For two irrational numbers $\alpha$ and $\beta$ satisfying $\alpha \pm \beta \not\in\Z$ there are only finitely many $\alpha_\nu^* = \beta_\mu^*$.
\end{lemma}
%the following lemma in fact was used in [5].
 \begin{figure}[ht]
	\centering
	\def\svgwidth{9cm} % используем для изменения размера, если надо
 \executeiffilenewer{figures/pic2.svg}{figures/pic2.pdf}%
  {inkscape -z -D --file=figures/pic2.svg %
 --export-pdf=figures/pic2.pdf --export-latex}%
  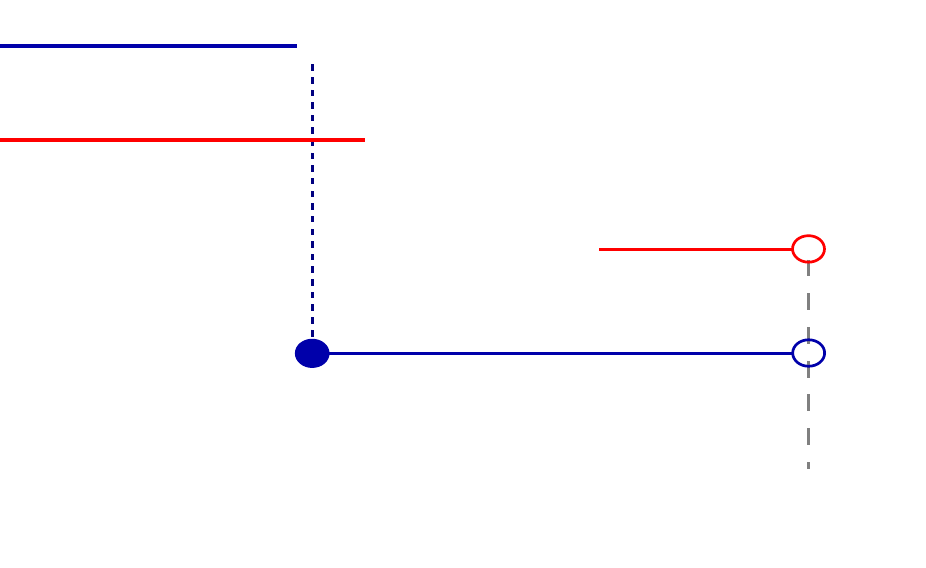
	\caption{to Remark \ref{rm:main}}
	\label{pic:Remark}
\end{figure}
The following lemma was proven in \cite{MM}. Here we formulate it without a proof.
\begin{lemma}
	\label{lm:MM}
	Let $d \geq 1$. Assume that the inequalites
	\begin{align}
		\label{eq:lmMM1}
		\xi_\nu &\leq \eta_\mu, \\
		\label{eq:lmMM2}
		\xi_{\nu+1} &\leq \eta_{\mu+d-1}, \\
		\label{eq:lmMM3}
		q_{\nu+1} &\leq r_{\mu+1}, \\
		q_{\nu+2} &= r_{\mu+d} \nonumber
	\end{align}
	hold. Then in (\ref{eq:lmMM1}), (\ref{eq:lmMM2}), (\ref{eq:lmMM3}) we have equalities, and
	\begin{gather}
	d=2, \\
	\alpha_{\mu+2}^* = \beta_{\mu+2}^*.
	\end{gather}
\end{lemma}
Let $t \in \R_+$. For $n$-tuple $\blalpha=(\alpha_1, \dots, \alpha_n)$ of pairwise independent irrational numbers let
$$\tau(t) = \Big| \left\{j \in \{1, \dots, n \}: \func{j} \text{ is discontinuous at } t \right\}\Big|$$
be the number of discontinuous functions at $t$.
%добавить сужение перестановки. более формально опеределить
Set $$ \func{j}(T-1) = \lim\limits_{t \rightarrow T-}\func{j}(t)$$ and 
$$ \blsigma(T-1) = \blsigma\{\func{1}(T-1), \dots, \func{n}(T-1)\}.$$
\begin{remark}
	\label{rm:main}
	Consider two independent irrational numbers $\alpha, \beta$. From Lemma \ref{lm:independent} and Lemma~\ref{lm:MM} we see that if
	\begin{gather*}
		q_{\nu} = r_{\mu} \text{ and } \psi_\alpha(q_{\nu}-1) < \psi_\beta(r_\mu-1),
	 \end{gather*}
	 then
	 \begin{equation*}
	 	\psi_\alpha(q_{\nu-1}-1) > \psi_\beta(q_{\nu-1}-1)
	 \end{equation*}
	 for all $\nu, \mu$ large enough. This situation is illustrated in Figure \ref{pic:Remark}. 
\end{remark}
Since there is only a finite number of permutations, there exists $T_0$ such that
\begin{equation}
	\label{eq:sigma repution}
	\forall t_0 > T_0 \ \exists t > t_0: \blsigma(t) = \blsigma(t_0).
\end{equation}
 \begin{figure}
	\centering
	\def\svgwidth{15cm} % используем для изменения размера, если надо
 \executeiffilenewer{figures/pic1.svg}{figures/pic1.pdf}%
  {inkscape -z -D --file=figures/pic1.svg %
 --export-pdf=figures/pic1.pdf --export-latex}%
  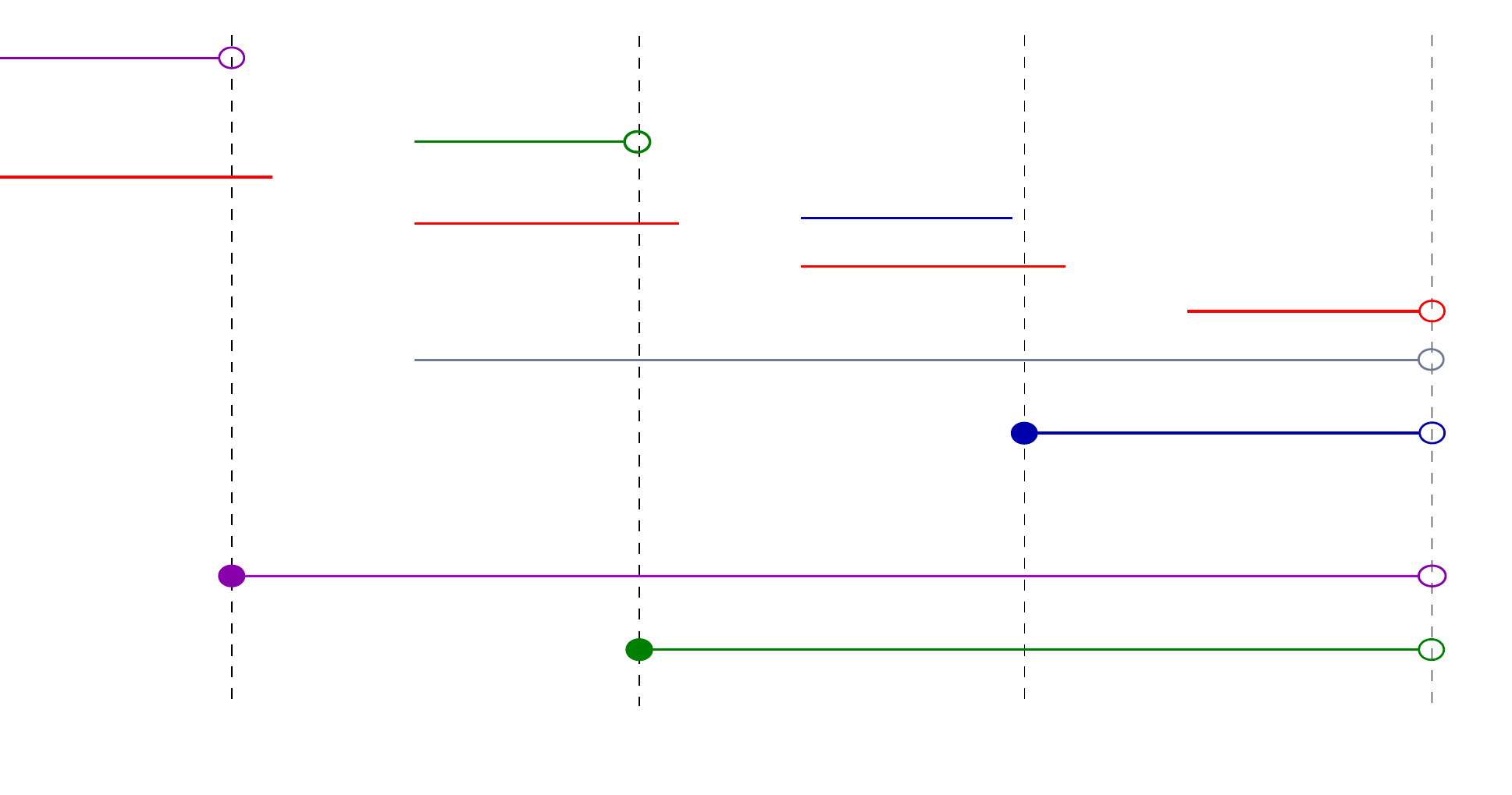%

	\caption{to Lemma \ref{lm:main}}
	\label{pic:Lemma}
\end{figure}
Our next auxiliary result is the following
\begin{lemma}
\label{lm:main}
    For all $t$ large enough one has $$\tau(t) \leq \knm(\blalpha).$$
\end{lemma}
\begin{proof}
	We consider arbitary $k \geq 2$ satisfying the condition
	\begin{equation*}
		\forall t_0 > 0 \ \exists t > t_0: \tau(t) = k.
	\end{equation*}
	Set $T > T_0$ large enough with $\tau(T) = k$. Let $\{\func{j_i}\}_{i=1}^k$ be the functions that are discontinuous at $T$. Since a finite collection of functions $\{\func{i}\}_{i=1}^n$ has finite number of $k$-element subsets, the set
	\begin{equation*}
		\T = \{t \geq T: \{\func{j_i}\}_{i=1}^k \text{ are discontinuous at } t\}
	\end{equation*}
	is infinite. Without loss of generality assume that $\{\func{j}\}_{j=1}^k$ are the functions that are discontinuous at $\T$. Let
	\begin{equation*}
		q^j_{l_j} = T, \ 1 \leq j \leq k 
	\end{equation*}
	be the denominators of convergents to $\alpha_j$ with $1 \leq j \leq k$.
	Since $T$ is large enough we may assume by Lemma \ref{lm:pairs} that
	\begin{equation*}
			q_{l_j-1}^j \neq q_{l_i-1}^i \text{ for } i \neq j.
	\end{equation*}
	Also we can assume, by reordering $ \{\func{j}\}_{j=1}^k$ if necessary, that 
	\begin{equation}
		\label{eq:start}
		\func{1}(T-1) > \func{i}(T-1) \text{ for all } 2 \leq i \leq k,
	\end{equation}
	and
	\begin{equation}
		\label{eq:denom}
		q_{l_2 - 1}^2 > q_{l_{3} - 1}^{3} > \ldots > q_{l_n - 1}^n.
	\end{equation}
	Denote 
	\begin{equation*}
		\blsigma_1 = \blsigma(T - 1), \ \blsigma_j = \blsigma(q_{l_j-1}^j-1), \ 2 \leq j \leq k.
	\end{equation*}
	From Remark \ref{rm:main} we see that for any $2 \leq i \leq k$
	\begin{equation}
		\label{eq:change}
		\func{1}(q_{l_i-1}^i - 1) < \func{i}(q_{l_i-1}^i - 1).
	\end{equation}
	and from (\ref{eq:start})
	\begin{equation}
		\label{eq:psi 1 > psi i}
		\func{1}(t) > \func{i}(t) \text{ for } t \in [q_{l_i-1}^i, T).
	\end{equation}
	This result is visualised in Figure \ref{pic:Lemma}.\\
	Consequently, for $2 \leq j \leq k$ one has
	\begin{equation*}
		\blsigma_j \neq \blsigma_i \text{ for } 1 \leq i < j
	\end{equation*}
	because of (\ref{eq:denom}), (\ref{eq:change}) and (\ref{eq:psi 1 > psi i}).
	So for $\tau(T) = k$ we have at least $k$ different permutations. From (\ref{eq:k defin}) and (\ref{eq:sigma repution}) one has
	\begin{equation*}
		k \leq \knm(\blalpha).
	\end{equation*}
	Since $k$ is arbitary, this completes the proof of Lemma \ref{lm:main}.
\end{proof}

\section{Proof of Theorem \ref{th:main}}

	%ввести остальные n перестановок, сказать что они будут встречатся в такой последовательности??
	Let $T > T_0$. Without loss of generality we can assume that 
	\begin{equation*}
		\blsigma(T) = (1, 2, 3, \dots, n),
	\end{equation*}
	 that is
	 \begin{equation*}
	 	\func{1}(T) > \func{2}(T) > \dots > \func{n}(T).
	 \end{equation*}
	Let $\blsigma_1 = \blsigma(T)$ and $T_1 = T$. Now we define inductivly values $T_2, \dots, T_k$ by the following relation
	\begin{equation}
		\label{eq:def I_j}
		T_j = \min\{ t > T_1: \blsigma(t) \neq \blsigma_i, \ 1 \leq i < j \}, \ 2 \leq j \leq k,
	\end{equation}
	where
	\begin{equation}
	\label{eq:sigma_def}
		\blsigma_i = \blsigma(T_i) \text{ for } 2 \leq i < j.
	\end{equation}
	Now $\blsigma_1, \dots, \blsigma_k$ are all different permutations which occurs infinitely many times.\\
	For $2 \leq j \leq k$ set
	\begin{multline}
		\label{eq:index}
		I_j = \{m \in \{1, \dots, n\}: \func{m}(t) \text{ is discontinuous at } T_j \\ \text{ and continuous at } T_i \text{ for all } 2 \leq i < j\},	
	\end{multline}
	in particular
	\begin{equation*}
		I_2 = \{m \in \{1, \dots, n\}: \func{m}(t) \text{ is discontinuous at } T_2  \}.
	\end{equation*}
	Let	$n_j$ be the number of elements in the set $I_j, \ 2 \leq j \leq k$.	
	For a given $j$  consider the collection of functions $\{\func{i}\}_{i \in I_j}$. Let
	\begin{gather*}
	 \widetilde{\blsigma}_s = \blsigma_s\left.\right|_{\{\alpha_i\}_{i \in I_j}}, \ 1 \leq s \leq k
	\end{gather*}
	be permutations restricted on the collection $\{\func{i}\}_{i \in I_j}$ (see Figure \ref{pic:Theorem}).
	Now we will show that
	\begin{equation}
		\label{eq:sigma_eq}
		\widetilde{\blsigma}_1 = \widetilde{\blsigma}_2 = \dots = \widetilde{\blsigma}_{j-1}.
	\end{equation}
	Indeed, suppose that for some $2 \leq i < j$ we have
	\begin{equation}
		\label{eq:sigma neq}
		\widetilde{\blsigma}_i \neq \widetilde{\blsigma}_1.
	\end{equation}
	It is clear that there are different $s, m \in I_j$ such that
	\begin{equation*}
		\func{m}(T_1) > \func{s}(T_1)
	\end{equation*}
	and
	\begin{equation*}
		\func{m}(T_i) < \func{s}(T_i).
	\end{equation*}
	Set
	\begin{equation}
		t_0 = \min \{ T_1<t \leq T_i: \func{m}(t) < \func{s}(t)  \}.
	\end{equation}
	So, $\func{m}$ is discontinuous at $t_0$ and
	\begin{equation}
		\label{eq:contr}
		\blsigma(t_0) \neq \blsigma(t) \text{ for } t \in [T_1, t_0).
	\end{equation}
	From (\ref{eq:def I_j}), (\ref{eq:sigma_def}) and (\ref{eq:contr}) we see that $t_0 \in\{T_l\}_{l=2}^{i}$ and $\func{m}$ is discontinuous at $T_l$ for some $2 \leq l < j$, which contradicts definition (\ref{eq:index}). So the assumption (\ref{eq:sigma neq}) is false and (\ref{eq:sigma_eq}) is proved.  
	
	Consider $\knm$-index for the collection of $n_j$ numbers $\{\alpha_i\}$. We see that
	\begin{equation*}
		\knm(\{\alpha_i\}_{i \in I_j}) \leq k - (j - 1) + 1 = k - j + 2
	\end{equation*}
	because of (\ref{eq:sigma_eq}).
	By Lemma \ref{lm:main} we have
	\begin{equation*}
		n_j \leq \knm(\{\alpha_i\}_{i \in I_j}),
	\end{equation*}
	 so
	\begin{equation}
		\label{eq:main k}
		n_j \leq k - j + 2. %добавить ли тут 2 \leq j \leq k
	\end{equation}	
	 \begin{figure}
		\centering
		\def\svgwidth{15cm}
 \executeiffilenewer{figures/pic3.svg}{figures/pic3.pdf}%
  {inkscape -z -D --file=figures/pic3.svg %
 --export-pdf=figures/pic3.pdf --export-latex}%
  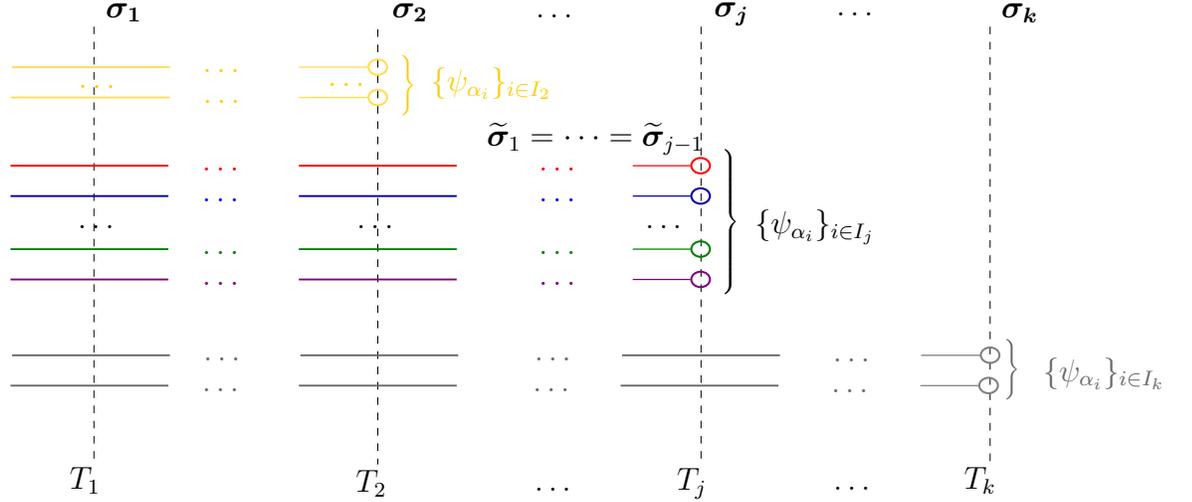%

		\caption{to Theorem \ref{th:main}}
		\label{pic:Theorem}
	\end{figure}
%	Consider the minimal function $\func{n}(t)$ at the moment $t=T_1$. According to Theorem \ref{th:KM} for each function $\func{j}(t), \ 1 \leq j < n$ there exists
%	\begin{gather*}
%		t_j = \min\{ t > T_1:  \func{j}(t) < \func{n}(t) \}.
%	\end{gather*} 
%	Hence $\func{j}(t)$ is discontinuous at $t_j$ and
%	\begin{equation*}
%		\blsigma(t_j) \neq \blsigma(t) \text{ for } t \in [T_1, t_j).
%	\end{equation*} 
	All of the possible permutations (which occur infinitely many times) are listed among $\blsigma_1, \dots, \blsigma_k$. By Theorem \ref{th:KM} we also should have a change of a sign in difference of irrationality measure functions for every pair $(\alpha_j, \alpha_n)$ with $1 \leq j < n$. This means a change should have happened in some $T_j$. Taking into account the definition of $I_j$ we get that
	\begin{equation*}
		j \in \bigsqcup\limits_{i=2}^k I_i \text{ for all } 1 \leq j < n
	\end{equation*}
	(here we have a disjoint union of the sets $I_i$).
	
	This gives
	\begin{equation*}
		n - 1 \leq \left|\bigsqcup\limits_{i=2}^k I_i\right| = \sum\limits_{i=2}^k n_i
	\end{equation*}
	and finally from (\ref{eq:main k}) we have
	\begin{equation*}
		n \leq 1 + \sum\limits_{i=2}^k n_i \leq \frac{k(k+1)}{2}.
	\end{equation*}
	The proof of Theorem \ref{th:main} is complete.
	\begin{ackn}
		The author thanks Nikolay Moshchevitin for many helpful discussions and support.
	\end{ackn}

\end{document}

%% file: thesys.tex
For $n$-tuple $\blalpha = (\alpha_1, \dots, \alpha_n)$ of pairwise independent numbers  we consider permutations  $$\blsigma(t): \{1, 2, 3, \dots, n\} \rightarrow \{ \sigma_1, \sigma_2, \sigma_3, \dots, \sigma_n \},$$  $$\func{\sigma_1}(t) > \func{\sigma_2}(t) > \func{\sigma_3}(t) > \dots > \func{\sigma_n}(t)$$ of irrationality measure functions $\func{}(t) = \min\limits_{1 \leq q \leq t} || q\alpha||$. Let $\knm(\blalpha)$ be the number of infinitely occurring different permutations $\{\blsigma_1, \dots, \blsigma_{\knm(\blalpha)}\}$.
We prove that the length of $n$-tuple $\blalpha$ with $\knm(\blalpha) = k$ is
$$ n \leq \frac{k(k+1)}{2}.$$
This result is optimal.

%Вопрос по тезису, нужно ли его делать, 
%For n-tuple of pairwise independent real numbers  we consider permutations of irrationality measure functions. 
%Let k be the number of infinitely occurring different permutations.
%We prove that the length of n-tuple with k infinitely occurring different permutations is k(k+1)/2.
%This result is optimal.

%% file: figures/pic2.pdf_tex
%% Creator: Inkscape 1.1.2 (b8e25be833, 2022-02-05), www.inkscape.org
%% Accompanies image file 'pic2.pdf' (pdf, eps, ps)
\begingroup%
  \makeatletter%
  \providecommand\color[2][]{%
    \errmessage{(Inkscape) Color is used for the text in Inkscape, but the package 'color.sty' is not loaded}%
    \renewcommand\color[2][]{}%
  }%
  \providecommand\transparent[1]{%
    \errmessage{(Inkscape) Transparency is used (non-zero) for the text in Inkscape, but the package 'transparent.sty' is not loaded}%
    \renewcommand\transparent[1]{}%
  }%
  \providecommand\rotatebox[2]{#2}%
  \newcommand*\fsize{\dimexpr\f@size pt\relax}%
  \newcommand*\lineheight[1]{\fontsize{\fsize}{#1\fsize}\selectfont}%
  \ifx\svgwidth\undefined%
    \setlength{\unitlength}{266.55312606bp}%
    \ifx\svgscale\undefined%
      \relax%
    \else%
      \setlength{\unitlength}{\unitlength * \real{\svgscale}}%
    \fi%
  \else%
    \setlength{\unitlength}{\svgwidth}%
  \fi%
  \global\let\svgwidth\undefined%
  \global\let\svgscale\undefined%
  \makeatother%
  \begin{picture}(1,0.61365267)%
    \lineheight{1}%
    \setlength\tabcolsep{0pt}%
    \put(0,0){\includegraphics[width=\unitlength,page=1]{pic2.pdf}}%
    \put(0.91430823,0.21788125){\color[rgb]{0,0,0.66666667}\makebox(0,0)[lt]{\lineheight{1.25}\smash{\begin{tabular}[t]{l}$\psi_{\alpha}$\end{tabular}}}}%
    \put(0.91172928,0.33598209){\color[rgb]{1,0,0}\makebox(0,0)[lt]{\lineheight{1.25}\smash{\begin{tabular}[t]{l}$\psi_{\beta}$\end{tabular}}}}%
    \put(0.80542921,0.05758945){\color[rgb]{0,0,0}\makebox(0,0)[lt]{\lineheight{1.25}\smash{\begin{tabular}[t]{l}$q_\nu = r_\mu$\end{tabular}}}}%
    \put(0.28648889,0.16275467){\color[rgb]{0,0,0}\makebox(0,0)[lt]{\lineheight{1.25}\smash{\begin{tabular}[t]{l}$q_{\nu-1}$\end{tabular}}}}%
    \put(0.4956781,0.38757777){\color[rgb]{1,0,0}\makebox(0,0)[lt]{\lineheight{1.25}\smash{\begin{tabular}[t]{l}$\ddots$\end{tabular}}}}%
    \put(0,0){\includegraphics[width=\unitlength,page=2]{pic2.pdf}}%
  \end{picture}%
\endgroup%

%% file: figures/pic1.pdf_tex
%% Creator: Inkscape 1.1.2 (b8e25be833, 2022-02-05), www.inkscape.org
%% Accompanies image file 'pic1.pdf' (pdf, eps, ps)
\begingroup%
  \makeatletter%
  \providecommand\color[2][]{%
    \errmessage{(Inkscape) Color is used for the text in Inkscape, but the package 'color.sty' is not loaded}%
    \renewcommand\color[2][]{}%
  }%
  \providecommand\transparent[1]{%
    \errmessage{(Inkscape) Transparency is used (non-zero) for the text in Inkscape, but the package 'transparent.sty' is not loaded}%
    \renewcommand\transparent[1]{}%
  }%
  \providecommand\rotatebox[2]{#2}%
  \newcommand*\fsize{\dimexpr\f@size pt\relax}%
  \newcommand*\lineheight[1]{\fontsize{\fsize}{#1\fsize}\selectfont}%
  \ifx\svgwidth\undefined%
    \setlength{\unitlength}{557.624579bp}%
    \ifx\svgscale\undefined%
      \relax%
    \else%
      \setlength{\unitlength}{\unitlength * \real{\svgscale}}%
    \fi%
  \else%
    \setlength{\unitlength}{\svgwidth}%
  \fi%
  \global\let\svgwidth\undefined%
  \global\let\svgscale\undefined%
  \makeatother%
  \begin{picture}(1,0.52583151)%
    \lineheight{1}%
    \setlength\tabcolsep{0pt}%
    \put(0,0){\includegraphics[width=\unitlength,page=1]{pic1.pdf}}%
    \put(0.96697991,0.22667648){\color[rgb]{0,0,0.66666667}\makebox(0,0)[lt]{\lineheight{1.25}\smash{\begin{tabular}[t]{l}$\psi_{\alpha_2}$\end{tabular}}}}%
    \put(0.96919123,0.0909865){\color[rgb]{0,0.50196078,0}\makebox(0,0)[lt]{\lineheight{1.25}\smash{\begin{tabular}[t]{l}$\psi_{\alpha_3}$\end{tabular}}}}%
    \put(0.96919125,0.13974265){\color[rgb]{0.4,0,0.50196078}\makebox(0,0)[lt]{\lineheight{1.25}\smash{\begin{tabular}[t]{l}$\psi_{\alpha_k}$\end{tabular}}}}%
    \put(0.96574713,0.31573664){\color[rgb]{1,0,0}\makebox(0,0)[lt]{\lineheight{1.25}\smash{\begin{tabular}[t]{l}$\psi_{\alpha_1}$\end{tabular}}}}%
    \put(0.66092585,0.03287562){\color[rgb]{0,0,0}\makebox(0,0)[lt]{\lineheight{1.25}\smash{\begin{tabular}[t]{l}$q_{l_2-1}^2$\end{tabular}}}}%
    \put(0.86267464,0.03287562){\color[rgb]{0,0,0}\makebox(0,0)[lt]{\lineheight{1.25}\smash{\begin{tabular}[t]{l}$T = q_{l_1}^1=q_{l_2}^2 $\end{tabular}}}}%
    \put(0.40960202,0.03193662){\color[rgb]{0,0,0}\makebox(0,0)[lt]{\lineheight{1.25}\smash{\begin{tabular}[t]{l}$q_{l_3-1}^3$\end{tabular}}}}%
    \put(0.13182301,0.031173){\color[rgb]{0,0,0}\makebox(0,0)[lt]{\lineheight{1.25}\smash{\begin{tabular}[t]{l}$q_{l_k-1}^k$\end{tabular}}}}%
    \put(0.63215971,0.50797217){\color[rgb]{0,0,0}\makebox(0,0)[lt]{\lineheight{1.25}\smash{\begin{tabular}[t]{l}$\boldsymbol{\sigma_2}$\end{tabular}}}}%
    \put(0.37057783,0.50869646){\color[rgb]{0,0,0}\makebox(0,0)[lt]{\lineheight{1.25}\smash{\begin{tabular}[t]{l}$\boldsymbol{\sigma_3}$\end{tabular}}}}%
    \put(0.89966975,0.50802401){\color[rgb]{0,0,0}\makebox(0,0)[lt]{\lineheight{1.25}\smash{\begin{tabular}[t]{l}$\boldsymbol{\sigma_1}$\end{tabular}}}}%
    \put(0.10813211,0.510714){\color[rgb]{0,0,0}\makebox(0,0)[lt]{\lineheight{1.25}\smash{\begin{tabular}[t]{l}$\boldsymbol{\sigma_k}$\end{tabular}}}}%
    \put(0,0){\includegraphics[width=\unitlength,page=2]{pic1.pdf}}%
    \put(0.27345933,0.03137441){\color[rgb]{0,0,0}\makebox(0,0)[lt]{\lineheight{1.25}\smash{\begin{tabular}[t]{l}$\dots$\end{tabular}}}}%
    \put(0.86794525,0.16990844){\color[rgb]{0.3254902,0.36470588,0.42352941}\makebox(0,0)[lt]{\lineheight{1.25}\smash{\begin{tabular}[t]{l}$\dots$\end{tabular}}}}%
    \put(0.86794525,0.11476382){\color[rgb]{0.3254902,0.36470588,0.42352941}\makebox(0,0)[lt]{\lineheight{1.25}\smash{\begin{tabular}[t]{l}$\dots$\end{tabular}}}}%
    \put(0.86929022,0.26136783){\color[rgb]{0.3254902,0.36470588,0.42352941}\makebox(0,0)[lt]{\lineheight{1.25}\smash{\begin{tabular}[t]{l}$\dots$\end{tabular}}}}%
    \put(0.35415873,0.40124689){\color[rgb]{0,0,0.50196078}\makebox(0,0)[lt]{\lineheight{1.25}\smash{\begin{tabular}[t]{l}$\dots$\end{tabular}}}}%
    \put(0.06767566,0.4268017){\color[rgb]{0.3254902,0.36470588,0.42352941}\makebox(0,0)[lt]{\lineheight{1.25}\smash{\begin{tabular}[t]{l}$\dots$\end{tabular}}}}%
    \put(0.06767565,0.46042648){\color[rgb]{0,0.50196078,0}\makebox(0,0)[lt]{\lineheight{1.25}\smash{\begin{tabular}[t]{l}$\dots$\end{tabular}}}}%
    \put(0.26518662,0.51097868){\color[rgb]{0,0,0}\makebox(0,0)[lt]{\lineheight{1.25}\smash{\begin{tabular}[t]{l}$\dots$\end{tabular}}}}%
    \put(0.45894364,0.37188817){\color[rgb]{1,0,0}\rotatebox{-17.939216}{\makebox(0,0)[lt]{\lineheight{1.25}\smash{\begin{tabular}[t]{l}. . .\end{tabular}}}}}%
    \put(0.7170152,0.34349254){\color[rgb]{1,0,0}\rotatebox{-17.939216}{\makebox(0,0)[lt]{\lineheight{1.25}\smash{\begin{tabular}[t]{l}. . .\end{tabular}}}}}%
    \put(0.46432516,0.40551263){\color[rgb]{0,0,0.66666667}\rotatebox{-17.939216}{\makebox(0,0)[lt]{\lineheight{1.25}\smash{\begin{tabular}[t]{l}. . .\end{tabular}}}}}%
    \put(0.19604948,0.39744311){\color[rgb]{1,0,0}\rotatebox{-16.8611381}{\makebox(0,0)[lt]{\lineheight{1.25}\smash{\begin{tabular}[t]{l}. . .\end{tabular}}}}}%
    \put(0.20158937,0.32094168){\color[rgb]{0.57647059,0.61568627,0.6745098}\rotatebox{-18.65075783}{\makebox(0,0)[lt]{\lineheight{1.25}\smash{\begin{tabular}[t]{l}. . .\end{tabular}}}}}%
    \put(0.19461457,0.45662238){\color[rgb]{0,0.50196078,0}\rotatebox{-16.8611381}{\makebox(0,0)[lt]{\lineheight{1.25}\smash{\begin{tabular}[t]{l}. . .\end{tabular}}}}}%
    \put(0,0){\includegraphics[width=\unitlength,page=3]{pic1.pdf}}%
    \put(0.20005617,0.23434464){\color[rgb]{0.43529412,0.48627451,0.56862745}\rotatebox{-18.65075783}{\makebox(0,0)[lt]{\lineheight{1.25}\smash{\begin{tabular}[t]{l}. . .\end{tabular}}}}}%
    \put(0.06767564,0.37838205){\color[rgb]{0.3254902,0.36470588,0.42352941}\makebox(0,0)[lt]{\lineheight{1.25}\smash{\begin{tabular}[t]{l}$\dots$\end{tabular}}}}%
    \put(0.86348439,0.00505071){\color[rgb]{0,0,0}\makebox(0,0)[lt]{\lineheight{1.25}\smash{\begin{tabular}[t]{l}$=\dots=q_{l_k}^k$\end{tabular}}}}%
  \end{picture}%
\endgroup%

%% file: figures/pic3.pdf_tex
%% Creator: Inkscape 1.1.2 (b8e25be833, 2022-02-05), www.inkscape.org
%% Accompanies image file 'pic3.pdf' (pdf, eps, ps)
\begingroup%
  \makeatletter%
  \providecommand\color[2][]{%
    \errmessage{(Inkscape) Color is used for the text in Inkscape, but the package 'color.sty' is not loaded}%
    \renewcommand\color[2][]{}%
  }%
  \providecommand\transparent[1]{%
    \errmessage{(Inkscape) Transparency is used (non-zero) for the text in Inkscape, but the package 'transparent.sty' is not loaded}%
    \renewcommand\transparent[1]{}%
  }%
  \providecommand\rotatebox[2]{#2}%
  \newcommand*\fsize{\dimexpr\f@size pt\relax}%
  \newcommand*\lineheight[1]{\fontsize{\fsize}{#1\fsize}\selectfont}%
  \ifx\svgwidth\undefined%
    \setlength{\unitlength}{557.624579bp}%
    \ifx\svgscale\undefined%
      \relax%
    \else%
      \setlength{\unitlength}{\unitlength * \real{\svgscale}}%
    \fi%
  \else%
    \setlength{\unitlength}{\svgwidth}%
  \fi%
  \global\let\svgwidth\undefined%
  \global\let\svgscale\undefined%
  \makeatother%
  \begin{picture}(1,0.52583151)%
    \lineheight{1}%
    \setlength\tabcolsep{0pt}%
    \put(0,0){\includegraphics[width=\unitlength,page=1]{pic3.pdf}}%
    \put(0.66997484,0.03335129){\color[rgb]{0,0,0}\makebox(0,0)[lt]{\lineheight{1.25}\smash{\begin{tabular}[t]{l}$T_j$\end{tabular}}}}%
    \put(0.92454422,0.03489325){\color[rgb]{0,0,0}\makebox(0,0)[lt]{\lineheight{1.25}\smash{\begin{tabular}[t]{l}$T _k$\end{tabular}}}}%
    \put(0.38539217,0.03193677){\color[rgb]{0,0,0}\makebox(0,0)[lt]{\lineheight{1.25}\smash{\begin{tabular}[t]{l}$T_2$\end{tabular}}}}%
    \put(0.131823,0.03453563){\color[rgb]{0,0,0}\makebox(0,0)[lt]{\lineheight{1.25}\smash{\begin{tabular}[t]{l}$T_1$\end{tabular}}}}%
    \put(0.7030131,0.45322159){\color[rgb]{0,0,0}\makebox(0,0)[lt]{\lineheight{1.25}\smash{\begin{tabular}[t]{l}$\boldsymbol{\sigma_j}$\end{tabular}}}}%
    \put(0.41741715,0.45299484){\color[rgb]{0,0,0}\makebox(0,0)[lt]{\lineheight{1.25}\smash{\begin{tabular}[t]{l}$\boldsymbol{\sigma_2}$\end{tabular}}}}%
    \put(0.95744617,0.45279791){\color[rgb]{0,0,0}\makebox(0,0)[lt]{\lineheight{1.25}\smash{\begin{tabular}[t]{l}$\boldsymbol{\sigma_k}$\end{tabular}}}}%
    \put(0.16391543,0.45295409){\color[rgb]{0,0,0}\makebox(0,0)[lt]{\lineheight{1.25}\smash{\begin{tabular}[t]{l}$\boldsymbol{\sigma_1}$\end{tabular}}}}%
    \put(0,0){\includegraphics[width=\unitlength,page=2]{pic3.pdf}}%
    \put(0.5438025,0.03271954){\color[rgb]{0,0,0}\makebox(0,0)[lt]{\lineheight{1.25}\smash{\begin{tabular}[t]{l}$\dots$\end{tabular}}}}%
    \put(0.54514751,0.45235668){\color[rgb]{0,0,0}\makebox(0,0)[lt]{\lineheight{1.25}\smash{\begin{tabular}[t]{l}$\dots$\end{tabular}}}}%
    \put(0.64198707,0.26607547){\color[rgb]{0,0,0}\makebox(0,0)[lt]{\lineheight{1.25}\smash{\begin{tabular}[t]{l}$\dots$\end{tabular}}}}%
    \put(0,0){\includegraphics[width=\unitlength,page=3]{pic3.pdf}}%
    \put(0.38643876,0.26607549){\color[rgb]{0,0,0}\makebox(0,0)[lt]{\lineheight{1.25}\smash{\begin{tabular}[t]{l}$\dots$\end{tabular}}}}%
    \put(0.80876563,0.03339204){\color[rgb]{0,0,0}\makebox(0,0)[lt]{\lineheight{1.25}\smash{\begin{tabular}[t]{l}$\dots$\end{tabular}}}}%
    \put(0.81145563,0.45437416){\color[rgb]{0,0,0}\makebox(0,0)[lt]{\lineheight{1.25}\smash{\begin{tabular}[t]{l}$\dots$\end{tabular}}}}%
    \put(0.73907251,0.26134036){\color[rgb]{0,0,0}\makebox(0,0)[lt]{\lineheight{1.25}\smash{\begin{tabular}[t]{l}$\{ \psi_{\alpha_i}\}_{i \in I_j}$\end{tabular}}}}%
    \put(0.5478375,0.31651263){\color[rgb]{1,0,0}\makebox(0,0)[lt]{\lineheight{1.25}\smash{\begin{tabular}[t]{l}$\dots$\end{tabular}}}}%
    \put(0.5478375,0.28961281){\color[rgb]{0,0,1}\makebox(0,0)[lt]{\lineheight{1.25}\smash{\begin{tabular}[t]{l}$\dots$\end{tabular}}}}%
    \put(0,0){\includegraphics[width=\unitlength,page=4]{pic3.pdf}}%
    \put(0.24991851,0.40393702){\color[rgb]{1,0.8,0}\makebox(0,0)[lt]{\lineheight{1.25}\smash{\begin{tabular}[t]{l}$\dots$\end{tabular}}}}%
    \put(0.2501611,0.37703721){\color[rgb]{1,0.8,0}\makebox(0,0)[lt]{\lineheight{1.25}\smash{\begin{tabular}[t]{l}$\dots$\end{tabular}}}}%
    \put(0.5478375,0.24253813){\color[rgb]{0,0.50196078,0}\makebox(0,0)[lt]{\lineheight{1.25}\smash{\begin{tabular}[t]{l}$\dots$\end{tabular}}}}%
    \put(0.5478375,0.21563831){\color[rgb]{0.50196078,0,0.50196078}\makebox(0,0)[lt]{\lineheight{1.25}\smash{\begin{tabular}[t]{l}$\dots$\end{tabular}}}}%
    \put(0,0){\includegraphics[width=\unitlength,page=5]{pic3.pdf}}%
    \put(0.13896046,0.26607549){\color[rgb]{0,0,0}\makebox(0,0)[lt]{\lineheight{1.25}\smash{\begin{tabular}[t]{l}$\dots$\end{tabular}}}}%
    \put(0.13963294,0.39048715){\color[rgb]{1,0.8,0}\makebox(0,0)[lt]{\lineheight{1.25}\smash{\begin{tabular}[t]{l}$\dots$\end{tabular}}}}%
    \put(0.36088398,0.39183215){\color[rgb]{1,0.8,0}\makebox(0,0)[lt]{\lineheight{1.25}\smash{\begin{tabular}[t]{l}$\dots$\end{tabular}}}}%
    \put(0,0){\includegraphics[width=\unitlength,page=6]{pic3.pdf}}%
    \put(0.24924952,0.31651263){\color[rgb]{1,0,0}\makebox(0,0)[lt]{\lineheight{1.25}\smash{\begin{tabular}[t]{l}$\dots$\end{tabular}}}}%
    \put(0.24924952,0.28961281){\color[rgb]{0,0,1}\makebox(0,0)[lt]{\lineheight{1.25}\smash{\begin{tabular}[t]{l}$\dots$\end{tabular}}}}%
    \put(0.24924952,0.24253813){\color[rgb]{0,0.50196078,0}\makebox(0,0)[lt]{\lineheight{1.25}\smash{\begin{tabular}[t]{l}$\dots$\end{tabular}}}}%
    \put(0,0){\includegraphics[width=\unitlength,page=7]{pic3.pdf}}%
    \put(0.25059429,0.14838911){\color[rgb]{0.4,0.4,0.4}\makebox(0,0)[lt]{\lineheight{1.25}\smash{\begin{tabular}[t]{l}$\dots$\end{tabular}}}}%
    \put(0.24924952,0.1214892){\color[rgb]{0.4,0.4,0.4}\makebox(0,0)[lt]{\lineheight{1.25}\smash{\begin{tabular}[t]{l}$\dots$\end{tabular}}}}%
    \put(0,0){\includegraphics[width=\unitlength,page=8]{pic3.pdf}}%
    \put(0.54380269,0.14838907){\color[rgb]{0.4,0.4,0.4}\makebox(0,0)[lt]{\lineheight{1.25}\smash{\begin{tabular}[t]{l}$\dots$\end{tabular}}}}%
    \put(0.54245792,0.12148916){\color[rgb]{0.4,0.4,0.4}\makebox(0,0)[lt]{\lineheight{1.25}\smash{\begin{tabular}[t]{l}$\dots$\end{tabular}}}}%
    \put(0.80851954,0.14716766){\color[rgb]{0.4,0.4,0.4}\makebox(0,0)[lt]{\lineheight{1.25}\smash{\begin{tabular}[t]{l}$\dots$\end{tabular}}}}%
    \put(0.80717477,0.12026775){\color[rgb]{0.4,0.4,0.4}\makebox(0,0)[lt]{\lineheight{1.25}\smash{\begin{tabular}[t]{l}$\dots$\end{tabular}}}}%
    \put(0.24924952,0.21563831){\color[rgb]{0.50196078,0,0.50196078}\makebox(0,0)[lt]{\lineheight{1.25}\smash{\begin{tabular}[t]{l}$\dots$\end{tabular}}}}%
    \put(0,0){\includegraphics[width=\unitlength,page=9]{pic3.pdf}}%
    \put(0.99462073,0.13020374){\color[rgb]{0.4,0.4,0.4}\makebox(0,0)[lt]{\lineheight{1.25}\smash{\begin{tabular}[t]{l}$\{ \psi_{\alpha_i}\}_{i \in I_k}$\end{tabular}}}}%
    \put(0.45191695,0.38709702){\color[rgb]{1,0.8,0}\makebox(0,0)[lt]{\lineheight{1.25}\smash{\begin{tabular}[t]{l}$\{ \psi_{\alpha_i}\}_{i \in I_2}$\end{tabular}}}}%
    \put(0.50168161,0.33896843){\color[rgb]{0,0,0}\makebox(0,0)[lt]{\lineheight{1.25}\smash{\begin{tabular}[t]{l}$\widetilde{\boldsymbol{\sigma}}_1 = \dots = \widetilde{\boldsymbol{\sigma}}_{j-1}$\end{tabular}}}}%
    \put(0,0){\includegraphics[width=\unitlength,page=10]{pic3.pdf}}%
  \end{picture}%
\endgroup%